\documentclass[12 pt]{article}
\usepackage{amsmath, amsfonts, amsthm, color,latexsym}
\usepackage[all]{xy}
\usepackage{amssymb}
\usepackage{color, soul}
\usepackage{graphicx}
\usepackage[all]{xy}

\allowdisplaybreaks[4]
\usepackage[all]{xy}
\providecommand{\U}[1]{\protect\rule{.1in}{.1in}}

\newtheorem{theo}{\setcounter{equation}{1}\setcounter{ctheo}{0}}[section]

\newenvironment{df}{\begin{theo} \bf Definition. \rm}{\end{theo}}

\newenvironment{lem}{\begin{theo} \bf Lemma. \it}{\end{theo}}
\newenvironment{teon}[1]{\begin{theo} \bf Theorem#1. \it}{\end{theo}}
\newenvironment{teo}{\begin{theo} \bf Theorem. \it}{\end{theo}}
\newenvironment{pro}{\begin{theo} \bf  Proposition. \it}{\end{theo}}
\newenvironment{cor}{\begin{theo} \bf Corollary. \it}{\end{theo}}

\begin{document}

\title{Spaces of $\sigma(p)$-nuclear linear and multilinear operators and their duals }
\author{Geraldo Botelho\thanks{Supported by CNPq Grant
305958/2014-3 and Fapemig Grant PPM-00490-15.}~ and Ximena Mujica\thanks{Supported
by a CNPq Postdoctoral scholarship.\thinspace \hfill\newline\indent2010 Mathematics Subject
Classification: 47L22, 46G25, 47B10, 46A20.\newline\indent Key words: $\sigma(p)$-nuclear operators, quasi-$\tau(p)$-summing operators, duality, Borel transform.}}
\date{}
\maketitle

\begin{abstract}
The theory of $\tau$-summing and $\sigma$-nuclear linear operators on Banach spaces was developed by Pietsch \cite[Chapter 23]{pietsch}. Extending the linear case to the range $p >1$ and generalizing all cases to the multilinear setting, in this paper we introduce the concept of $\sigma(p)$-nuclear linear and multilinear operators. In order to develop the duality theory for the spaces of such operators, we introduce the concept of quasi-$\tau(p)$-summing linear/multilinear operators and prove Pietsch-type domination theorems for such operators. The main result of the paper shows that, under usual conditions, linear functionals on the space of $\sigma(p)$-nuclear $n$-linear operators are represented, via the Borel transform, by quasi-$\tau(p)$-summing $n$-linear operators. As far as we know, this result is new even in the linear case $n=1$.
\end{abstract}

\section{Introduction} \setcounter{equation}{0}

Nuclear and absolutely summing linear operators, investigated by A. Grothendieck in the 1950's, turned out to be the germs of several important classes of operators between Banach spaces. These classes play a central role in the theory of operator ideals systematized by A. Pietsch in the 1970's \cite{pietsch}. Among such classes one can find $\sigma$-nuclear and $\tau$-summing operators (see, e.g., \cite[Chapter 23]{pietsch}).
In \cite{xmtDom}  the concept of $\tau$-summing linear operators was generalized to $\tau(p)$-summing linear and multilinear operators for $p \geq 1$. Several types of nuclear linear operators have already been generalized to the multilinear setting (see, e.g., \cite{matos, popa}), and, naturally enough, in Section 2 of this paper we extend the notion of $\sigma$-nuclear linear operator to $\sigma(p)$-nuclear linear and multilinear operators, $p \geq 1$ (cf. Definition \ref{fdef}). For $p > 1$ this concept is new even in the linear case. We prove that the class of $\sigma(p)$-nuclear multilinear operators is a Banach multi-ideal, in particular the class of $\sigma(p)$-nuclear linear operators is a Banach operator ideal. As a preparation for the next section, we prove, under usual conditions on the underlying spaces, a simpler formula for the $\sigma(p)$-nuclear norm of a finite type operator.

In Section 3 we develop the duality theory for the class of $\sigma(p)$-nuclear operators. The standard tool to represent linear functionals on classes of multilinear operators is the Borel transform (see, e.g., Dineen \cite[Chapter 2]{dineen}). Given a class ${\cal M}_1$ of multilinear operators, usually the problem is the identification of the class ${\cal M}_2$ of multilinear operators such that linear functionals on operators belonging to ${\cal M}_1$ are isometrically represented, via the Borel transform, by operators belonging to ${\cal M}_2$. The search for a class of operators that represent bounded linear functionals on the space of $\sigma(p)$-nuclear linear/multilinear operators led us to the introduction of the class of quasi-$\tau(p)$-summing linear/multilinear operators (cf. Definition \ref{sdef}). We prove that such operators enjoy a Pietsch-type domination theorem and the main result asserts that, under usual conditions, the Borel transform is an isometric isomorphism between the dual of the space of $\sigma(p)$-nuclear $n$-linear operators from $E_1 \times \cdots \times E_n$ to $F$ and the space of $\tau(p)$-summing $n$-linear operators from $E_1' \times \cdots \times E_n'$ to $F'$ (cf. Theorem \ref{mth}). We stress that this result is new even in the linear case $n = 1$.

The symbols $D, D_i, E, E_i, F, F_i, G, G_i$ stand for
Banach spaces over $ {\mathbb K} = \mathbb{R}$ or $\mathbb{C}$, $B_E$ is the closed unit ball of $E$ and $E'$ is the (topological) dual of $E$.
For $p \geq 1$, we denote by $p^\prime $ its conjugate, that is, $1 =
\frac{1}{p}+\frac{1}{p^\prime }$. By ${\cal L}(E_1,
\ldots, E_n;F)$ we denote the Banach space of all continuous $n$-linear operators from
$E_1  \times \cdots  \times E_n$ into $F$ endowed with the usual sup norm. Given $n \in \mathbb{N}$, $k \in \mathbb{N}\cup\{\infty\}$,  consider sequences $(\lambda_j)_{j=1}^k $    in $
\mathbb{K}$,  $ (x^\prime_{ij})_{j=1}^k $ in $
  E_i^\prime $ for $i = 1, \ldots, n$, and $(y_{j})_{j=1}^k  $  in $
 F$. If an $n$-linear operator $ A \in {\cal L}(E_1, \ldots, E_n;F)$ is such that
$$A(x_1, \ldots, x_n) = \sum_{j=1}^k \lambda_j x^\prime_{1j}(x_1)\cdots  x^\prime_{nj}(x_n)y_j $$
for all $x_1 \in E_1, \ldots, x_n \in E_n$, then we shall write $$ A =  \sum\limits_{j=1}^k  \lambda_j   x^\prime_{1j}\otimes \cdots \otimes x^\prime_{nj}\otimes y_{j} .$$
If $k \in \mathbb{N}$ we say that $A$ is a {\it finite type operator} and write $A \in  {\cal L}_f(E_1, \ldots, E_n;F)$. The linear case $n=1$ recovers that space of finite rank linear operators. For background on spaces of multilinear operators we refer to \cite{dineen, mujica}.

We shall denote by  ${\cal L}_n$  the class of all continuous $n$-linear operators between Banach spaces. The definition of a Banach ideal of $n$-linear operators can be found in, e.g., \cite{klauspams, klausdomingo}. Instead of the definition, we shall use the following well known characterization:

\begin{teo}{} \label{teo1}
Let ${\cal M}$  be a subclass of ${\cal L}_n$  endowed with a non-negative function  $\|\cdot\|_{\cal M} \colon {\cal M} \longrightarrow \mathbb{R}$. For Banach spaces $E_1, \ldots, E_n,F$, define
$${\cal M}(E_1, \ldots, E_n;F) := {\cal M} \cap {\cal L}(E_1, \ldots, E_n;F). $$
Then $({\cal M}, \|\cdot\|_{\cal M})$ is a Banach ideal of $n$-linear operators if and only if the following conditions hold:

\noindent {\rm (i)} The $n$-linear operator $I_n \colon \mathbb{K}^n \rightarrow \mathbb{K}$, $I_n(\lambda_1, \ldots, \lambda_n)  :=  \lambda_1 \ldots \lambda_n$ , belongs to
 $ {\cal M} $ and $\|I_n\|_{\cal M} = 1$.

\noindent {\rm (ii)} If $S_1,S_2, \ldots \in {\cal M}(  E_1, \ldots, E_n;$ $F)$  and  $ \sum\limits_{k=1}^\infty
\|S_k\|_{\cal M} < \infty$,  then  $S :=  \sum\limits_{k=1}^\infty   S_k \in {\cal M}(  E_1, \ldots,
$ $E_n;$ $F)$  and  $\|S\|_{\cal M} \leq  \sum\limits_{k=1}^\infty  \|S_k\|_{\cal M}$.

\noindent  {\rm (iii)} If $T_i \in {\cal L}(D_i;E_i)$, $i = 1, \ldots,n$, $S\in {\cal M}(  E_1,\ldots,E_n;F)$  and  $R
\in {\cal L}(F;G)$, then $R \circ S \circ (T_1, \ldots,T_n) \in {\cal M}(  D_1,$ $\ldots,$
$D_n;$ $G)$  and  $\|R\circ S \circ (T_1, \ldots, T_n)\|_{\cal M} \leq \|R\|\cdot \|S\|_{\cal M} \cdot \|T_1\|\ldots
\|T_n\|$.
\end{teo}

Every Banach ideal $\cal M$ of $n$-linear operators contains the finite type $n$-linear operators 
$$\|x_1'\otimes \cdots \otimes x_n' \otimes b\|_{\cal M} = \|x_1'\| \cdots \|x_n'\| \cdot \|b\|. $$


\section{ $\sigma(p)$-nuclear $n$-linear operators}\label{spnuc} \setcounter{equation}{0}

By definition (see \cite[Section 23.1]{pietsch}), an operator $u \in {\cal L}(E;F)$ is $\sigma$-nuclear if there are sequences $(x_j')_{j=1}^\infty$ in $E'$ and $(y_j)_{j=1}^\infty$ in $F$ such that $u = \sum\limits_{j=1}^\infty x_j' \otimes y_j$ and the sequence $(x_j' \otimes y_j)_{j=1}^\infty$ is unconditionally summable in ${\cal L}(E;F)$. The following characterization fits our purposes:

\begin{pro}\label{fpro} An operator $u \in {\cal L}(E;F)$ is $\sigma$-nuclear if and only if there are sequences $(\lambda_j)_{j=1}^\infty \in \ell_\infty$, $(x_j')_{j=1}^\infty$ in $E'$ and $(y_j)_{j=1}^\infty$ in $F$ such that $u = \sum\limits_{j=1}^\infty \lambda_jx_j' \otimes y_j$,
$$ \sup_{{x \in B_{E}}\atop {y^\prime \in B_{  F^\prime }}}  \sum_{j=1}^\infty  |x^\prime_{j}(x) y^\prime (y_j)|  <
\infty {\rm ~and~} \lim_{m  \rightarrow \infty}  \sup_{{x \in B_{E}}\atop {y^\prime \in B_{  F^\prime }}}  \sum_{j=m}^\infty  |x^\prime_{j}(x) y^\prime (y_j)| = 0. $$
\end{pro}

\begin{proof} By \cite[Propositions 8.3 and 8.1]{df}, a sequence $(x_j' \otimes y_j)_{j=1}^\infty$ is unconditionally summable in ${\cal L}(E;F)$ if and only if
$$ \sup_{\varphi \in D}  \sum_{j=1}^\infty  |\varphi(x^\prime_{j}\otimes y_j)|  <
\infty {\rm ~and~} \lim_{m  \rightarrow \infty}  \sup_{\varphi \in D}  \sum_{j=m}^\infty  |\varphi(x^\prime_{j}\otimes y_j)| = 0 $$
for some norming set $D \subseteq B_{{\cal L}(E;F)'}$.
For $x \in E$ and $y' \in F'$, the map
$$\varphi_{x,y'}\colon {\cal L}(E;F) \rightarrow \mathbb{K}~,~\varphi_{x,y'}(v) = y'(v(x)), $$
is a bounded linear functional. The (obvious) fact that the set $(\varphi_{x,y'})_{x \in B_E, y' \in B_{F'}} \subseteq B_{{\cal L}(E;F)'}$ is norming completes the proof.
\end{proof}

Now we are in position to generalize the class of $\sigma$-nuclear linear operators to the multilinear setting and to introduce the classes of $\sigma(p)$-nuclear linear and multilinear operators for $p > 1$.

\begin{df}\label{fdef}
{    For $1 \leq p < \infty$, we say that an
  $n$-linear operator
$A\colon  E_1 \times \cdots  \times E_n \rightarrow  F$
  is  {\it ${\sigma(p)}$-nuclear} if there are sequences $(\lambda_j)_{j=1}^\infty \in\ell_{p^\prime }$,  $ (x^\prime_{ij})_{j=1}^\infty $ in
  $ E_i^\prime$, for $i = 1, \ldots, n$, and $(y_{j})_{j=1}^\infty $ in $
 F$, such that $ A =  \sum\limits_{j=1}^\infty  \lambda_j   x^\prime_{1j}\otimes \cdots \otimes x^\prime_{nj}\otimes y_{j} $,

\begin{equation*}\label{sigp1}
 \sup_{{x_i \in B_{E_i}}\atop {y^\prime \in B_{  F^\prime }}} \Bigl( \sum_{j=1}^\infty  |x^\prime_{1j}(x_1)  \cdots   x^\prime_{nj}(x_n)  y^\prime (y_j)|^p\Bigr)^{1/p}  <
\infty   \end{equation*}  and
\begin{equation*}\label{sigp2}   \lim_{m  \rightarrow \infty}  \sup_{{x_i \in B_{E_i}}\atop {y^\prime \in B_{  F^\prime }}} \Bigl( \sum_{j=m}^\infty  |x^\prime_{1j}(x_1)  \cdots
x^\prime_{nj}(x_n)  y^\prime (y_j)|^p\Bigr)^{1/p} = 0.
\end{equation*}
In this case we say that $ A =  \sum\limits_{j=1}^\infty  \lambda_j   x^\prime_{1j} \otimes \cdots \otimes   x^\prime_{nj} \otimes y_{j} $ is a $\sigma(p)$-nuclear representation of $A$ and define
$$\|A\|_{\sigma(p)}{\color{blue}:}= \inf \biggl\{\|(\lambda_j)_{j=1}^\infty\|_{p^\prime }\cdot  \sup_{{x_i \in B_{E_i}}\atop {y^\prime \in B_{  F^\prime }}}  \Bigl( \sum_{j=1}^\infty  |x^\prime_{1j}(x_1)
 \cdots   x^\prime_{nj}(x_n)  y^\prime (y_j)|^p\Bigr)^{1/p}\biggr\}, $$
 where the infimum runs over all $\sigma(p)$-nuclear representations of $A$.
  The set of all such $n$-linear operators is denoted by $ {\cal L}_{\sigma(p)} (E_1,\ldots,E_n;F)$.
}
\end{df}

By Proposition \ref{fpro} the case $n=p=1$ recovers the Banach ideal of $\sigma$-nuclear linear operators from \cite[Section 23.1]{pietsch}.

\begin{pro}$[{\cal L}_{\sigma(p)} $,  $\|\cdot \|_{\sigma(p)}]
$  is  a Banach ideal of $n$-linear operators. In particular, the class of $\sigma(p)$-nuclear linear operators is a Banach operator ideal.
\end{pro}

\begin{proof} Let us prove condition (ii) of Theorem \ref{teo1} (the remaining conditions follow easily). First we remark that routine computations show that
 $ \|\cdot \|  \leq \|\cdot \|_{\sigma(p)}$ on ${\cal L}_{\sigma(p)}$.

Let $A_1, A_2, \ldots \in
 {\cal L}_{\sigma(p)} (E_1,\ldots,E_n;F)$ be such that
 $  \sum\limits_{j=1}^\infty  \| A_k \|_{\sigma(p)} < \infty$. Since $ \|\cdot \|  \leq \|\cdot \|_{\sigma(p)}$, the series
$ \sum\limits_{j=1}^\infty   A_k   $ is absolutely convergent in the Banach space  ${\cal L}(E_1, \ldots, E_n;F)$, therefore     $A:= \sum\limits_{k=1}^\infty A_k $ converges in ${\cal L}(E_1, \ldots, E_n;F)$.  \\

Now we shall see   $ A $ is $\sigma(p)$-nuclear. Given $\varepsilon > 0 $, for each $k$ take a $\sigma(p)$-nuclear representation
$ A_k=  \sum\limits_{j=1}^\infty   \lambda_{kj} {x'_{1kj}}  \otimes   \cdots  \otimes {x'_{nkj}}  \otimes y_{kj}$  such that
\begin{eqnarray*}
 \|( \lambda_{kj})_{j=1}^\infty \|_{p' }
  \leq  \Bigl[(1+   \varepsilon) \|A_k\|_{\sigma(p)}\Bigr]^{1/p' }    {\rm ~and}
 \end{eqnarray*}
 \begin{eqnarray*}
 \sup_{{x_i \in B_{E_i}}\atop {y' \in B_{  F' }}}  \Bigl(  \sum_{j=1}^\infty  |x'_{1kj} (x_1) \cdots  {x'_{nkj}} (x_n)
y'(y_{kj})|^p\Bigr)^{1/p}
 \leq   \Bigl[(1+   \varepsilon)  \|A_k\|_{\sigma(p)} \Bigr]^{1/p}.     \label{Ak}
\end{eqnarray*}
Let us see that $A =\sum\limits_{k=1}^\infty \sum\limits_{j=1}^\infty   \lambda_{kj} {x'_{1kj}}  \otimes   \cdots  \otimes {x'_{nkj}}  \otimes y_{kj} $ is a $\sigma(p)$-nuclear representation of $A$. We have
\begin{align*} \|( \lambda_{kj})_{j,k=1}^\infty \|_{p' }^{p'}  =   \sum_{k=1}^\infty   \sum_{j=1}^\infty
| \lambda_{kj}|^{p' }
          \leq     \sum_{k=1}^\infty  \biggl(  \Bigl[ (1+   \varepsilon)
\|A_k\|_{\sigma(p)} \Bigr]^{1/p' } \biggr)^{p' }
 = \ (1+  \varepsilon)\cdot \sum_{k=1}^\infty    \|A_k\|_{\sigma(p)} < \infty,
\end{align*}
 and
\begin{align*}
\sup_{{x_i \in B_{E_i}}\atop {y' \in B_{  F' }}}    \sum_{k=1}^\infty   \sum_{j=1}^\infty
|{x'_{1kj}} (x_1) \cdots {x'_{nkj}} (x_n) y'(y_{kj})|^p
 \leq  & \sum_{k=1}^\infty    \sup_{{x_i \in B_{E_i}}\atop {y' \in B_{  F' }}}     \sum_{j=1}^\infty  |{x'_{1kj}} (x_1)
\cdots {x'_{nkj}} (x_n) y'(y_{kj})|^{p }     \nonumber  \\
\leq  &   (1+   \varepsilon)\cdot   \sum_{k=1}^\infty    \|A_k\|_{\sigma(p)} < \infty.      \label{lim00}
  \end{align*}
To check the condition concerning the tail of the series, let $\delta > 0$ be given.
Observing that, for each $m \in \mathbb{N}$,
\begin{eqnarray*}
A  & := & \sum_{k=1}^\infty A_k =
     \sum_{k=1}^\infty   \sum_{j=1}^\infty   \lambda_{kj} {x'_{1kj}}  \otimes   \cdots \otimes {x'_{nkj}} \otimes y_{kj}  \\
 & = &  \sum_{k=1}^{m-1} \sum_{j=1}^{m-1} \lambda_{kj} {x'_{1kj}}  \otimes   \cdots \otimes {x'_{nkj}} \otimes y_{kj}  +  \\
&  & +\underbrace{
 \sum_{k=1}^{m-1}  \sum_{j=m }^\infty   \lambda_{kj} {x'_{1kj}}  \otimes   \cdots \otimes {x'_{nkj}} \otimes y_{kj}
    +  \sum_{k=m}^\infty   \sum_{j=1}^\infty   \lambda_{kj} {x'_{1kj}}  \otimes   \cdots \otimes {x'_{nkj}} \otimes y_{kj}
}_{\textrm{tail}  }
\end{eqnarray*}
we have to show that there is $M \in \mathbb{N}$ such that
\begin{eqnarray}
\lefteqn{ \sup_{{x_i \in B_{E_i}}\atop {y' \in B_{  F' }}}   \biggl\{  \Bigl(
  \sum_{k=1}^{m-1}  \sum_{j=m}^\infty    |{x'_{1kj}}(x_1)  \cdots  {x'_{nkj}}(x_n)  y' (y_{kj})|^p   + }
  \hspace{30mm}  \nonumber  \\
  &&  +\sum_{k=m}^\infty   \sum_{j=1}^\infty    |{x'_{1kj}}(x_1)  \cdots  {x'_{nkj}}(x_n)  y' (y_{kj})|^p
   \Bigr)^{1/p}    \biggr\}
    < \delta   \nonumber
    \end{eqnarray}
for every $m \geq M$. By $  \sum\limits_{j=1}^\infty  \| A_k \|_{\sigma(p)} < \infty$,  there exists $K_\delta \in {\mathbb N} $ such that
$  \sum\limits_{k=K_\delta}^\infty        \| A_k \|_{\sigma(p)} < \frac{\delta^p}{2(1+\varepsilon)} $. Hence
\begin{align*}
  \sup_{{x_i \in B_{E_i}}\atop {y' \in B_{  F' }}}    \sum_{k=K_\delta}^\infty   \sum_{j=1}^\infty
|{x'_{1kj}} (x_1) &\cdots {x'_{nkj}} (x_n) y'(y_{kj})|^p
 \leq   \sum_{k=K_\delta}^\infty     \sup_{{x_i \in B_{E_i}}\atop {y' \in B_{  F' }}}     \sum_{j=1}^\infty  |{x'_{1kj}} (x_1)
\cdots {x'_{nkj}} (x_n) y'(y_{kj})|^{p }     \nonumber  \\
     &\leq     (1+   \varepsilon)\cdot      \sum_{k=K_\delta}^\infty     \|A_k\|_{\sigma(p)}    \leq      \frac{\delta^p}{2}   .
 \end{align*}
 For $k = 1, \ldots, K_\delta -1$, since $ \sum\limits_{j=1}^\infty   \lambda_{kj} {x'_{1kj}}  \otimes   \cdots  \otimes {x'_{nkj}}  \otimes y_{kj}$ is a $\sigma(p)$-nuclear representation of $A_k$,
there is $J_k \in \mathbb{N}$ such that
\begin{equation*}\label{lim000}
\sup_{{x_i \in B_{E_i}}\atop {y' \in B_{  F' }}}
  \sum_{j=J_k}^\infty  |x'_{1kj} (x_1) \cdots  {x'_{nkj}} (x_n)y'(y_{kj})|^p
 \ \leq \frac{\delta^p}{2^{k+1}} .
\end{equation*}
 Choosing $M =  \max\{K_\delta , J_1, $
  $ \ldots, J_{K_\delta-1}\}$, we have for $m \geq M$:
  \begin{align*}
  \sup_{{x_i \in B_{E_i}}\atop {y' \in B_{  F' }}}    &    \biggl\{
     \sum_{k=1}^{m-1}   \sum_{j=m}^\infty    |{x'_{1kj}} (x_1) \cdots {x'_{nkj}} (x_n) y'(y_{kj})|^p
+    \sum_{k=m}^\infty   \sum_{j=1}^\infty   |{x'_{1kj}} (x_1) \cdots {x'_{nkj}} (x_n) y'(y_{kj})|^p \biggr\}    \\ 
= & \sup_{{x_i \in B_{E_i}}\atop {y' \in B_{  F' }}}     \biggl\{
 \sum_{k=1}^{M-1}   \sum_{j=m}^\infty  |{x'_{1kj}} (x_1) \cdots {x'_{nkj}} (x_n) y'(y_{kj})|^p     +     \\    
&
     +\sum_{k=M}^{m-1}   \sum_{j=m}^\infty  |{x'_{1kj}} (x_1) \cdots {x'_{nkj}} (x_n) y'(y_{kj})|^p
+
  \sum_{k=m}^\infty   \sum_{j=1}^\infty   |{x'_{1kj}} (x_1) \cdots {x'_{nkj}} (x_n) y'(y_{kj})|^p \biggr\}
\\ 
\leq & \sup_{{x_i \in B_{E_i}}\atop {y' \in B_{  F' }}}     \biggl\{
 \sum_{k=1}^{M-1}   \sum_{j=M}^\infty  |{x'_{1kj}} (x_1) \cdots {x'_{nkj}} (x_n) y'(y_{kj})|^p     +     \\    
&
     +\sum_{k=M}^{m-1}   \sum_{j=1}^\infty  |{x'_{1kj}} (x_1) \cdots {x'_{nkj}} (x_n) y'(y_{kj})|^p
+
  \sum_{k=m}^\infty   \sum_{j=1}^\infty   |{x'_{1kj}} (x_1) \cdots {x'_{nkj}} (x_n) y'(y_{kj})|^p \biggr\}
\\ 
=  & \sup_{{x_i \in B_{E_i}}\atop {y' \in B_{  F' }}}     \biggl\{
 \sum_{k=1}^{M-1}   \sum_{j=M}^\infty  |{x'_{1kj}} (x_1) \cdots {x'_{nkj}} (x_n) y'(y_{kj})|^p     +     \\    
&        \hspace{30mm}
     +\sum_{k=M}^\infty   \sum_{j=1}^\infty   |{x'_{1kj}} (x_1) \cdots {x'_{nkj}} (x_n) y'(y_{kj})|^p \biggr\}    \\ 
\leq & \sup_{{x_i \in B_{E_i}}\atop {y' \in B_{  F' }}}     \biggl\{
 \sum_{k=1}^{K_\delta-1}   \sum_{j=M}^\infty  |{x'_{1kj}} (x_1) \cdots {x'_{nkj}} (x_n) y'(y_{kj})|^p     +    \\    
&
  +   \sum_{k=K_\delta}^{M-1}   \sum_{j=M}^\infty  |{x'_{1kj}} (x_1) \cdots {x'_{nkj}} (x_n) y'(y_{kj})|^p
+    \sum_{k=M}^\infty   \sum_{j=1}^\infty   |{x'_{1kj}} (x_1) \cdots {x'_{nkj}} (x_n) y'(y_{kj})|^p \biggr\}      \\ 
\leq & \sup_{{x_i \in B_{E_i}}\atop {y' \in B_{  F' }}}     \biggl\{
 \sum_{k=1}^{K_\delta-1}   \sum_{j=J_k}^\infty  |{x'_{1kj}} (x_1) \cdots {x'_{nkj}} (x_n) y'(y_{kj})|^p     +     \\    
&     +\sum_{k=K_\delta}^{M-1}   \sum_{j=1}^\infty  |{x'_{1kj}} (x_1) \cdots {x'_{nkj}} (x_n) y'(y_{kj})|^p
+      \sum_{k=M}^\infty   \sum_{j=1}^\infty   |{x'_{1kj}} (x_1) \cdots {x'_{nkj}} (x_n) y'(y_{kj})|^p \biggr\}      \\ 
\leq & \sup_{{x_i \in B_{E_i}}\atop {y' \in B_{  F' }}}     \biggl\{
 \sum_{k=1}^{K_\delta-1}   \sum_{j=J_k}^\infty  |{x'_{1kj}} (x_1) \cdots {x'_{nkj}} (x_n) y'(y_{kj})|^p     +     \\    
& \hspace{40mm}
+ \sum_{k=K_\delta}^\infty \sum_{j=1}^\infty |{x'_{1kj}} (x_1)\cdots {x'_{nkj}} (x_n) y'(y_{kj})|^p \biggr\}    \\ 
\leq &     \sup_{{x_i \in B_{E_i}}\atop {y' \in B_{  F' }}}     \sum_{k=1}^{K_\delta-1}
  \sum_{j=J_k}^\infty  |{x'_{1kj}} (x_1) \cdots {x'_{nkj}} (x_n) y'(y_{kj})|^p         +     \\    
&    \hspace{40mm}
      \sup_{{x_i \in B_{E_i}}\atop {y' \in B_{  F' }}}          \sum_{k=K_\delta}^\infty
     \sum_{j=1}^\infty   |{x'_{1kj}} (x_1) \cdots {x'_{nkj}} (x_n) y'(y_{kj})|^p      \\ 
\leq &    \sum_{k=1}^{K_\delta-1}    \sup_{{x_i \in B_{E_i}}\atop {y' \in B_{  F' }}}
  \sum_{j=J_k}^\infty  |{x'_{1kj}} (x_1) \cdots {x'_{nkj}} (x_n) y'(y_{kj})|^p         +     \\    
&    \hspace{40mm}
     +\sup_{{x_i \in B_{E_i}}\atop {y' \in B_{  F' }}}\sum_{k=K_\delta}^\infty
     \sum_{j=1}^\infty   |{x'_{1kj}} (x_1) \cdots {x'_{nkj}} (x_n) y'(y_{kj})|^p      \\ 
& \leq  \sum_{k=1}^{K_\delta-1}  \frac{\delta^p}{2^{k+1}}  +   \frac{ \delta^p}{2}  <     \delta^p   .  
 \end{align*}

This completes the proof that $A$ is $\sigma(p)$-nuclear. From
\begin{align}
\|( \lambda_{kj})_{j,k=1}^\infty \|_{p^\prime } \cdot
\sup_{{x_i \in B_{E_i}}\atop {y^\prime \in B_{  F^\prime }}}  \Bigl( & \sum_{k=1}^\infty   \sum_{j=1}^\infty  |{x^{\prime }_{1kj}}(x_1) \cdots  {x^{\prime }_{nkj}}(x_n)   y'(y_{kj})|^p \Bigr)^{1/p}    \nonumber \\
& \leq  (1+   \varepsilon)^{1/p^\prime } \cdot \Bigl[  \sum_{k=1}^\infty    \|A_k\|_{\sigma(p)}
\Bigr]^{1/p^\prime } \cdot (1+   \varepsilon)^{1/p} \cdot\Bigl[  \sum_{k=1}^\infty  \|A_k\|_{\sigma(p)}
\Bigr]^{1/p}    \nonumber \\
& =  (1+   \varepsilon)\cdot   \sum_{k=1}^\infty  \|A_k\|_{\sigma(p)},  \nonumber      
  \end{align}
letting $\varepsilon \rightarrow 0$, we conclude that $    \|A \|_{\sigma(p)}  \leq  \sum\limits_{k=1}^\infty  \|A_k\|_{\sigma(p)} . $
   \end{proof}

Once [$ {\cal L}_{\sigma(p)} $,  $\|\cdot\|_{\sigma(p)}]
$  is  a Banach ideal of $n$-linear operators, we have ${\cal L}_f(E_1,
\ldots,E_n;F) \subseteq {\cal L}_{\sigma(p)} (E_1, \ldots,E_n;F)$ and
$$ \| \lambda   x^\prime_1\otimes
\cdots \otimes   x^\prime_n \otimes y \|_{\sigma(p)} = | \lambda  |\cdot \| x^\prime_1\| \cdots
\|x^\prime_n\|\cdot\|y\|,$$
for all $\lambda \in \mathbb{K}, x_j' \in E_j', y \in F$. Using partial sums of $\sigma(p)$-nuclear representations it is easy to see that ${\cal L}_f(E_1,
\ldots,E_n;F)$ is $\|\cdot\|_{\sigma(p)}$-dense in ${\cal L}_{\sigma(p)} (E_1, \ldots,E_n;F)$.
For $A \in {\cal L}_f(E_1,
\ldots,E_n;F)$, define
$$\|A\|_{\sigma(p)f}= \inf \Biggl\{\|(\lambda_j)_{j=1}^m\|_{p^\prime } \cdot \sup_{{x_i \in B_{E_i}}\atop {y^\prime \in B_{  F^\prime }}}  \Bigl( \sum_{j=1}^m  |x^\prime_{1j}(x_1)  \cdots   x^\prime_{nj}(x_n)  y^\prime (y_j)|^p  \Bigr)^{1/p} \Biggr\} $$
where the infimum runs over all finite representations $A = \sum\limits_{j=1}^m \lambda_j\otimes x_{1,j}'\cdots \otimes x_{n,j}'\otimes y_j$.
\noindent  It is easy to check that $\|\cdot\|_{\sigma(p)f}$ is a norm on ${\cal L}_f(E_1, \ldots, E_n;F)$ such that $\|\cdot \| \leq \|\cdot\|_{\sigma(p)f}$. It follows that $\|\cdot\|_{\sigma(p)f}$ is a complete norm on ${\cal L}_f(E_1, \ldots, E_n;F)$. It   is  clear that $\|\cdot\|_{{\sigma(p)}}\leq
\|\cdot\|_{\sigma(p)f}$. For further use, we shall establish conditions under which the equality holds.

\begin{lem}\label{6.4.1} If the norms $ \|\cdot\|_{\sigma(p)}$ and
$\|\cdot\|_{\sigma(p)f} $ are equivalent on ${\cal L}_f(E_1, \ldots,E_n;F)$, then they coincide on this space.
\end{lem}

\begin{proof}
By assumption there is a constant $c> 0$ such that $\|\cdot \|_{\sigma(p)f}\leq c\,
\|\cdot\|_{{\sigma(p)}} $ on ${\cal L}_f(E_1, \ldots,E_n;F)$.
Given  $A \in
{\cal L}_f(E_1,\ldots,E_n;F)$  and  $\varepsilon >0$, take an infinite $\sigma(p)$-nuclear representation

$ A =  \sum\limits_{j=1}^\infty  \lambda_j   x^\prime_{1j}  \otimes   \cdots \otimes   x^\prime_{nj} \otimes y_j$

\noindent   such that
\begin{eqnarray*} \|(\lambda_j)_{j=1}^\infty\|_{p^\prime } \cdot \sup_{{x_i \in B_{E_i}}\atop {y^\prime \in B_{  F^\prime }}}  \Bigl( \sum_{j=1}^\infty
|x^\prime_{1j}(x_1) \cdots   x^\prime_{nj}(x_n)  y^\prime (y_j)|^p\Bigr)^{1/p} \leq
 \Bigl(1+\frac{\varepsilon}{2} \Bigr)\|A\|_{{\sigma(p)}}. \end{eqnarray*}
In particular, for each $m \in  {\mathbb N}$,
\begin{eqnarray*} \Bigl\| \sum_{j=1}^{m-1}\lambda_j   x^\prime_{1j}  \otimes   \cdots \otimes   x^\prime_{nj}   \otimes
y_j \Bigr\|_{\sigma(p)f} & \leq & \|(\lambda_j)_{j=1}^{m-1}\|_{p^\prime } \cdot \sup_{{x_i \in B_{E_i}}\atop {y^\prime \in B_{  F^\prime }}}
\Bigl( \sum_{j=1}^{m-1} |x^\prime_{1j}(x_1) \cdots   x^\prime_{nj}(x_n)
 y^\prime (y_j)|^p\Bigr)^{1/p} \nonumber \\
& \leq &  \Bigl(1+\frac{\varepsilon}{2} \Bigr) \|A\|_{\sigma(p)}. \label{eq:6.4.1c} \end{eqnarray*}
Since
$$\lim_{m  \rightarrow \infty}   \sup_{{x_i \in B_{E_i}}\atop {y^\prime \in B_{  F^\prime }}}  \Bigl(  \sum_{j=m}^\infty  |x^\prime_{1j}(x_1)  \cdots  x^\prime_{nj}(x_n)  y^\prime (y_j)|^p\Bigr)^{1/p} = 0,$$ for a sufficiently large $m \in  {\mathbb N}$ we get
\begin{align*}   \Bigl\|   \sum_{j=m}^\infty  \lambda_j   x^\prime_{1j}  \otimes   \cdots  \otimes     x^\prime_{nj}  \otimes
y_j  \Bigr\| _{\sigma(p)}
& \leq   \|(\lambda_j)_{j=m}^\infty\|_{p^\prime }\cdot
 \sup_{{x_i \in B_{E_i}}\atop {y^\prime \in B_{  F^\prime }}}  \Bigl(  \sum_{j=m}^\infty  |x^\prime_{1j}(x_1)  \ldots  x^\prime_{nj}(x_n)  y^\prime (y_j)|^p\Bigr)^{1/p}  \nonumber \\
& \leq  \|(\lambda_j)_{j=1}^\infty\|_{p^\prime }\cdot
 \sup_{{x_i \in B_{E_i}}\atop {y^\prime \in B_{  F^\prime }}}  \Bigl(  \sum_{j=m}^\infty  |x^\prime_{1j}(x_1)  \ldots  x^\prime_{nj}(x_n)  y^\prime (y_j)|^p\Bigr)^{1/p}  \nonumber \\
& \leq  \frac{\varepsilon}{2c}\|A\|_{\sigma(p)}.\label{eq:6.4.1d} \end{align*}
 It follows that
\begin{eqnarray*}
 \|A\|_{\sigma(p)f}  & \leq  &  \left\|
\sum_{j=1}^{m-1} \lambda_j   x^\prime_{1j} \otimes \cdots \otimes   x^\prime_{nj}   \otimes
y_j\right\|_{\sigma(p)f} +
\left\|  \sum_{j=m}^\infty  \lambda_j   x^\prime_{1j} \otimes \cdots \otimes   x^\prime_{nj} \otimes y_j \right\|_{\sigma(p)f}  \\
& \leq &  \Bigl(1+\frac{\varepsilon}{2} \Bigr) \|A\|_{\sigma(p)}+ c\,
\left\| \sum_{j=m}^\infty  \lambda_j   x^\prime_{1j}  \otimes   \cdots  \otimes     x^\prime_{nj}  \otimes   y_j\right\|_{\sigma(p)} \\
& \leq & \Bigl(1+\frac{\varepsilon}{2} \Bigr) \|A\|_{\sigma(p)}+ \frac{\varepsilon}{2} \|A\|_{\sigma(p)}
\ = \ (1 + \varepsilon ) \|A\|_{\sigma(p)}.
  \end{eqnarray*}
And as this holds for every  $\varepsilon >0$, the result follows.
       \end{proof}

\begin{lem}\label{6.4.2}  If $A \in   {\cal L}_{\sigma(p)} (E_1, \ldots,E_n;F)$  and  $T_k \in
{\cal L}_f(D_k;E_k)$, $k=1,\ldots,n$, then
$$\|A  \circ (T_1,\ldots,T_n)\|_{\sigma(p)f}\leq\|A \|_{{\sigma(p)}} \ \|T_1\|\ldots\|T_n\|. $$
\end{lem}

 \begin{proof} Letting $J_k \colon T_k(D_k)\rightarrow E_k$ be the formal inclusions and $ \tilde{T}_k\colon D_k \rightarrow T_k(D_k)$ be defined by $\tilde{T}_k(u_k) = {T}_k(u_k)$, we can write
  $T_k = J_k \circ \tilde{T}_k$. Since each $T_k(D_k)$ is finite dimensional, we have
  $${\cal L}_f(T_1(D_1), \ldots, T_n(D_n);F) = {\cal L}(T_1(D_1), \ldots, T_n(D_n);F) = {\cal L}_{\sigma(p)}(T_1(D_1), \ldots, T_n(D_n);F),  $$
  so ${\cal L}_f(T_1(D_1), \ldots, T_n(D_n);F)$ is complete with both
  norms $ \|\cdot\|_{\sigma(p)}$ and $ \|\cdot\|_{\sigma(p)f}$. By the inequality $\|\cdot \|_{\sigma(p)} \leq \|\cdot\|_{\sigma(p)f}$  and
the open mapping theorem we conclude that these norms are equivalent on ${\cal L}_f(T_1(D_1), \ldots, T_n(D_n);F)$.  By Lemma \ref{6.4.1} we get
$$\|A \circ (J_1,\ldots,J_n)\|_{\sigma(p)f}  =
\|A \circ (J_1,\ldots,J_n)\|_{{\sigma(p)}} \ \leq \ \|A\|_{{\sigma(p)}}\cdot
 \|J_1\|\cdots\|J_n\| \ = \ \|A\|_{{\sigma(p)}},$$
from which it follows that
\begin{align*} \|A \circ ({T}_1,\ldots,{T}_n)\|_{\sigma(p)f} & =
\|A \circ (J_1,\ldots,J_n) \circ (\tilde{T}_1,\ldots,\tilde{T}_n)\|_{\sigma(p)f} \\
&\leq
\|A \circ (J_1,\ldots,J_n)\|_{\sigma(p)f}\cdot \|\tilde{T}_1\|\cdots\|\tilde{T}_n\|=  \|A\|_{{\sigma(p)}}\cdot\|{T}_1\|\ldots\|{T}_n\|.
 \end{align*}
 \end{proof}

\begin{pro}\label{6.4.3} If $E_1^\prime,   \ldots, E_n^\prime $ have the bounded approximation property, then \linebreak $ \|\cdot\|_{\sigma(p)f} = \|\cdot\|_{{\sigma(p)}} $ on  ${\cal L}_f(E_1, \ldots, E_n;F)$ regardless of the Banach space $F$.
\end{pro}

\begin{proof} We give the proof for $n = 2$, as for other values of $n$ it is similar.  Let $\gamma_i \geq 1$ be such that $E_i^\prime $ has the $\gamma_i$-bounded approximation property for $i = 1,2$. Given $A \in {\cal L}_f(E_1,E_2;F)$, defining $A_1 \colon E_1 \rightarrow {\cal L}(E_2;F)$ and $A_2 \colon E_2 \rightarrow {\cal L}(E_1;F)$ by
$$A_1(x_1)(x_2) = A_2(x_2)(x_1) = A(x_1,x_2), $$   we have $A_1 \in {\cal L}_f(E_1;{\cal L}(E_2;F))$, $A_2 \in {\cal L}_f(E_2;{\cal L}(E_1;F))$ and $\|A_1\| = \|A_2\| =\|A\|$.
Given $ \varepsilon > 0 $, by \cite[Lemma 10.2.6]{pietsch}
there are  $T_i \in {\cal L}_f(E_i;E_i)$, $i = 1,2$, such that
$\|T_i\| \leq (1+\varepsilon )\gamma_i$ and  $A_i \circ T_i = A_i$.
Thus
\begin{align*}A(T_1(x_1), T_2(x_2)) &= [A_1\circ T_1(x_1)](T_2(x_2))= [A_1(x_1)](T_2(x_2))= A(x_1, T_2(x_2))\\
&=[A_2 \circ T_2(x_2)](x_1) = A_2(x_2)(x_1) = A(x_1, x_2)
\end{align*}
for all $x_i \in E_i$, proving that
 $A = A  \circ (T_1,T_2)$. Calling on Lemma \ref{6.4.2} we have
\begin{eqnarray*} \|A\|_{\sigma(p)f}  = \|A  \circ (T_1,T_2)\|_{\sigma(p)f} \leq  \|A\|_{\sigma(p)}\cdot \|T_1\|\cdot\|T_2\| \leq  (1 + \varepsilon )^2 \gamma_1\gamma_2\|A\|_{\sigma(p)}. \end{eqnarray*}

Letting $\varepsilon \rightarrow 0$ we get
$ \|A\|_{\sigma(p)f} \leq \gamma_1\gamma_2\|A\|_{\sigma(p)} $.
The result follows from Lemma \ref{6.4.1}.
\end{proof}


\section{The dual of $ {\cal L}_{\sigma(p)} (E_1,\ldots,E_n;F)$}\label{tpq} \setcounter{equation}{0}

Our aim is to represent bounded linear functionals on the space $ {\cal L}_{\sigma(p)} (E_1,\ldots,E_n;F)$ of $\sigma(p)$-nuclear multilinear operators. Since this space contains the finite type $n$-linear operators, the Borel transform
$$ {\cal B}\colon   [  {\cal L}_{\sigma(p)} (E_1, \ldots,  E_n;F), \|\cdot\|_{\sigma(p)}]'   \longrightarrow   {\cal L}(E^\prime_1, \ldots,  E^\prime_n;F^\prime ), $$
$${\cal B}(\varphi)(x_1', \ldots, x_n')(y) := \varphi(x_1' \otimes \cdots x_n'\otimes y), $$
is a well defined linear operator. The question is to identify the range of $\cal B$ and a norm on it that makes $\cal B$ an isometric isomorphism. The relation proved in \cite[Theorem 23.2.13]{pietsch} draws our attention to the class of $\tau(p)$-summing multilinear operators investigated by the second author in \cite{xmtDom} as a generalization of the class of $\tau$-summing linear operators. As we shall see later, this class works only if the Banach space $F$ is reflexive (cf. Corollary \ref{finalcor}). For the general case we need the following slightly larger class:

\begin{df}\label{sdef}
For  $1 \leq q \leq p$, an $n$-linear operator $S \in {\cal L}(E_1, \ldots, E_n;F^\prime )$   is said to be {\it quasi-$\tau(p;q)$-summing } if there is a constant $  {\sf C}  \geq 0$  such that
\begin{equation*}\label{qtpq} \Bigl( \sum_{j=1}^m | S(x_{1j},\ldots, x_{nj})(y_j)|^p\Bigr)^{1/p} \leq  {\sf C}   \sup_{ {x^\prime_i \in B_{E^\prime_i} } \atop {y^\prime \in B_{  F^\prime  } } } \Bigl( \sum_{j=1}^m  | x^\prime_1(x_{1j}) \ldots   x^\prime_n(x_{nj}) y^\prime (y_j)|^q\Bigr)^{1/q},
\end{equation*}
for all $m\in  {\mathbb N},\ x_{ij}\in E_i,\ y_j\in F,  $
$ i=1,2,\dots,n$, $j=1,2,\dots,m$. The infimum of all such constants $\sf C$ is denoted by $\|S\|_{q\tau(p;q)}$. We denote this class of operators by $ {{\cal L}_{q\tau(p;q)}}(E_1, \ldots,   E_n; F^\prime )$. Routine computations show that  $ [{{\cal L}_{q\tau(p;q)}}(E_1, \ldots,   E_n; F^\prime ), \|\cdot\|_{q\tau(p;q)}] $ is a Banach space.

Whenever $p=q$, we simply write ${\cal L}_{q\tau(p)}$ ,
$\|S\|_{q\tau(p)}$ \label{tp}
and    say $S$  is  quasi-$\tau(p)$-summing. If
 $p=q=1$, we write $  {{{\cal L}_{q\tau}}}  $,  $\|S\|_{q\tau}$ \label{t}
 and    say $S$  is  quasi-$\tau$-summing.  \end{df}

Denoting by $ {{\cal L}_{\tau(p;q)}} (E_1,\ldots,  E_n;
F')$ the space of $\tau(p;q)$-summing $n$-linear operators from \cite{xmtDom}, it is
 straightforward
that
$ {{\cal L}_{\tau(p;q)}} (E_1,\ldots,  E_n; F') \subseteq
{{\cal L}_{q\tau(p;q)}} (E_1,\ldots,  E_n; F')$ with
$\|\cdot\|_{q\tau(p;q)} \leq \|\cdot\|_{\tau(p;q)}$ for every $F$ and that
$ {{\cal L}_{q\tau(p;q)}} (E_1,\ldots,  E_n; F')= {{\cal L}_{\tau(p;q)}} (E_1,\ldots,  E_n; F')$ isometrically for reflexive $F$.


For further use, we show that quasi-$\tau(p)$-summing multilinear operators enjoy  Pietsch-type domination characterizations.

  \begin{teon}{}\label{qtauDom}  Let $1 \leq  p < \infty$. The following are equivalent for an $n$-linear operator
  $S \in {\cal L}(E_1,   \ldots,   E_n;F^\prime )$:\\
{\rm (a)} $S$ is quasi-$\tau(p)$-summing.\\
{\rm (b)} There exist a constant $  {\sf B}   >0$  and  a regular Borel probability measure $\mu$ on $B_{ E^\prime _1} \times \cdots  \times
B_{ E^\prime _n}  \times B_{F'}$ endowed with the product of the weak-star topologies,   such that
\begin{eqnarray*}\label{qtaudomb}
 | S(x_1, \ldots,   x_n)(y)| \leq  {\sf B}  \biggl(\int_{B_{ E^\prime_1  } \times \cdots     \times B_{ E^\prime_n  } \times B_{F^\prime }}
            | x^\prime_1(x_1) \ldots   x^\prime_n(x_n)  y^\prime (y)|^p
d\mu ( x^\prime_1, \ldots 	,  x^\prime_n, y^\prime )\biggr)^{1/p}, \nonumber
\end{eqnarray*}
\noindent  for all  $ x_i \in E_i$  and   $y \in F$.\\
{\rm (c)} There exist a constant $ {\sf C}
> 0$ and regular Borel probability measures $\mu_i$ on $B_{E^\prime_i}$, $i = 1, \ldots, n$, $\mu_{n+1}$ on $B_{F^\prime}$
 such that
 \begin{eqnarray*}\label{qtaudomc}
| (S(x_1, \ldots , x_n)(y)| \leq  {\sf C}\cdot \prod_{k=1}^n \left( \int_{B_{E^\prime_k}}|x_k'(x_k)|^p d\mu_k (x^\prime_k)\right)^{1/p}\cdot \left(\int_{B_{F^\prime}}|y'(y)|^p d\mu_{n+1}(y^\prime) \right)^{1/p},
\end{eqnarray*}
\noindent   for all  $ x_i \in E_i$ and   $y \in F$.\\
\indent In this case,
$\|S\|_{\tau(p)} = \inf  {\sf B}  = \inf  {\sf C} $.
\end{teon}

\begin{proof} (c) $\Rightarrow$ (b) Just take the product measure $\mu := \mu_1 \otimes \cdots \otimes \mu_{n+1}$.\\
(b) $\Rightarrow$ (a) Taking finite sums in the inequality in (b) to the power $p$, it follows easily that $S$ is quasi-$\tau(p)$-summing.\\
(a) $\Rightarrow$ (c) The proof is analogous to the proof of \cite[Theorem 2.6]{xmtDom}, we shall only sketch the main steps. For a Banach space $E$, by $W(B_{E^\prime})$ we denote the (compact) set of regular Borel probability measures on $B_{E'}$ endowed with the  weak* topology of $C(B_{E'})'$.
Given $x_{i1},\ldots,x_{im}\in E_i$, $i=1,\ldots,n$ and
$y_1,\ldots,y_m\in F $, the function $\phi \colon W(B_{E_1^\prime})\times \cdots \times W(B_{E_n^\prime}) \times W(B_{F^\prime}) \longrightarrow \mathbb{R}^+$ defined by
 \begin{eqnarray*} \lefteqn{\phi(\mu_1, \ldots, \mu_n, \mu_{n+1}) :=   \sum_{j=1}^m
\biggl\{| S(x_{1j}, \ldots ,x_{nj})(y_j)|^p - }\\
& &  \hspace{23mm}  {\sf C}^p  \cdot \prod_{k=1}^n \left(\int_{B_{E^\prime_k}}|x_k'(x_{kj})|^p d\mu_k (x^\prime_k)\right)\cdot \int_{B_{F^\prime}}|y'(y_j)|^p d\mu_{n+1}(y^\prime) \biggr\}
 \end{eqnarray*}
is continuous and convex.
By compactness we can choose  $x^\prime_{10}\in B_{E_1^\prime}, \ldots, $ $ \ x^\prime_{n0}\in
B_{E_n^\prime}$ and $ y^\prime_0 \in B_{F^\prime}$ such that
$$ \sup\biggl\{ \sum_{j=1}^m  |x^\prime_1(x_{1j})  \ldots  x^\prime_n(x_{nj})  y^\prime_n( y_j ) |^p:\| x^\prime_i \|, \| y^\prime \| \leq 1 \biggr\} =
 \sum_{j=1}^m  |x^\prime_{10}(x_{1j})  \ldots x^\prime_{n0}(x_{nj})  y^\prime_0( y_j)|^p. $$
Let $\delta_1(x^\prime_{10}),\ldots,
\delta_n(x^\prime_{n0}),\delta_{n+1}( y^\prime_0)$ be the Dirac measures at
$ x^\prime_{10},\ldots, x^\prime_{n0} , y^\prime_0 $ respectively. Since $S$ is quasi-$\tau(p)$-summing, we have
\begin{eqnarray*}
\lefteqn{\phi \left(\delta_1( x^\prime_{10}) , \ldots ,
\delta_n( x^\prime_{n0}) , \delta_{n+1}( y^\prime_0)\right)=  }  \\
& & \sum_{j=1}^m  | S( x_{1j} , \ldots , x_{nj} )(y_j) |^p -
{\sf C}^p  | x^\prime_{10}(x_{1j})  \cdots  x^\prime_{n0}(x_{nj})  y^\prime_0( y_j) |^p \leq 0 .
\end{eqnarray*}
Since the collection  ${\cal F}$  of all such functions $\phi$ is  concave, by Ky Fan's Lemma \cite[Lemma E.4.2]{pietsch} there exists an element
 $(\mu_{1} , \ldots , \mu_{n} , \mu_{n+1})
\in
 W(B_{E_1^\prime})\times \cdots \times W(B_{E_n^\prime}) \times W(B_{F^\prime})$  such that
$\phi ( \mu_{1} , \ldots , \mu_{n} , \mu_{n+1} ) \leq
0$ for all $\phi \in {\cal F}$. In particular, given $x_1 \in E_1, \ldots,  x_n \in E_n, y \in F $ and $y \in F$, consider the function $\phi$ associated to $x_1, \ldots, x_n, y$ (that is, $m = 1$), to get the desired inequality.
\end{proof}

It is noteworthy that, although (b) is apparently weaker than (c), these conditions are actually equivalent. On the one hand, condition (b) is the analogue of \cite[Theorem 23.1.6]{pietsch} (a direct short proof of it follows from the results of \cite{bprMona, bprJMAA}); on the other hand, it is condition (c) that  next shall be useful.

The class of $\tau$-summing linear operators ($\tau(1)$-summing operators in our terminology) is rather small in the sense that it is contained in the other classes of summing-type linear operators \cite[Proposition 23.1.5]{pietsch}. Next we compare the class of quasi-$\tau(p;q)$-summing operators with other types of summing multilinear operators, and in particular we show that, though formally larger than the class of $\tau(p)$-summing multilinear operators, the class of quasi-$\tau(p)$-summing operators is still rather small. For the classes ${\cal L}_{as(p;q_1, \ldots, q_n)}$ of absolutely $(p;q_1, \ldots, q_n)$-summing multilinear operators, ${\cal L}_{d,p}$ of $p$-dominated multilinear operators and ${\cal L}_{si(p)}$ of $p$-semi-integral multilinear operators, see, e.g. \cite{daniel/erhan, david}.

\begin{pro} {\rm (a)} If $\frac{1}{q}
\leq \frac{1}{q_1} + \cdots +\frac{1}{q_n}+ +\frac{1}{q_{n+1}}$ and $S \in   {{\cal L}_{q\tau(p;q)}}(E_1, \ldots,   E_n; F^\prime )$, then the $(n+1)$-linear operator
$$S_F \colon E_1 \times \cdots \times   E_n \times F \rightarrow {\mathbb K}~,~
 S_F(x_{1},\ldots, x_{n},y)=   S(x_{1},\ldots, x_{n})(y), $$
is absolutely $(p;q_1, \ldots, q_n, q_{n+1})$-summing.\\
{\rm (b)} The following hold for spaces of multilinear operators taking values in dual spaces:
$${\cal L}_{q\tau(p)} \subseteq {\cal L}_{d,p} \subseteq {\cal L}_{si(p)} \subseteq {\cal L}_{as(p;p, \ldots, p)}.$$
In particular, quasi-$\tau(p)$-summing linear operators are absolutely $p$-summing.
\end{pro}

 \begin{proof} (a)
The result follows from  H\"older's inequality:
\begin{eqnarray*}  \lefteqn{\hspace{-1mm}\Bigl( \sum_{j=1}^m |S_F(x_{1_j},\ldots, x_{n_j},y_j)|^p\Bigr)^{1/p}
 =  \Bigl( \sum_{j=1}^m | S(x_{1_j},\ldots, x_{n_j})(y_j )|^p\Bigr)^{1/p}} \\
&\leq &   \| S \|_{q\tau(p;q)} \cdot \sup_{{ x^\prime_i\in B_{E^\prime_i} }\atop{ y^\prime_i\in B_{F^\prime_i} } }
\Bigl( \sum_{j=1}^m | x^\prime_1(x_{1j}) \ldots   x^\prime_n(x_{nj}) y^\prime (y_j) |^q \Bigr)^{1/q} \\
& \leq  & \| S \|_{q\tau(p;q)}\cdot
\sup_{{ x^\prime_i\in B_{E^\prime_i} }\atop{ y^\prime_i\in B_{F^\prime_i} } }
\Bigl( \sum_{j=1}^m | x^\prime_1(x_{1j})^{q_1}\Bigr)^{1/q_1} \ldots
\Bigl( \sum_{j=1}^m |x^\prime_n(x_{nj})|^{q_n}\Bigr)^{1/q_n}
\Bigl( \sum_{j=1}^m |y^\prime_j(y_j)|^{q_{n+1}}\Bigr)^{1/{q_{n+1}}} .
 \end{eqnarray*}

\noindent (b) Let $S \in {\cal L}_{q\tau(p)}(E_1, \ldots, E_n;F)$ and let $\mu_1, \ldots, \mu_{n+1}$ be the corresponding measures given by Theorem \ref{qtauDom}(c). Hence,
\begin{align*}
 \| S(x_1, &\ldots,   x_n)\| =   \sup_{ {y \in B_{F}}}     | S(x_1, \ldots,   x_n)(y)|  \\ & \leq \| S \|_{q\tau(p)} \cdot\sup_{ {y \in B_{F}}}
    \prod_{k=1}^n \left(\int_{B_{E^\prime_k}}|x_k'(x_k)|^p d\mu_k (x^\prime_k)\right)^{1/p}\cdot \left(\int_{B_{F^\prime}}|y'(y)|^p d\mu_{n+1}(y^\prime) \right)^{1/p}\\ & \leq \| S \|_{q\tau(p)} \cdot\sup_{ {y \in B_{F}}}
    \prod_{k=1}^n \left(\int_{B_{E^\prime_k}}|x_k'(x_k)|^p d\mu_k (x^\prime_k)\right)^{1/p}\cdot \left(\int_{B_{F^\prime}}\|y'\|^p\cdot \|y\|^p d\mu_{n+1}(y^\prime) \right)^{1/p}\\
     & = \| S \|_{q\tau(p)} \cdot
    \prod_{k=1}^n \left(\int_{B_{E^\prime_k}}|x_k'(x_k)|^p d\mu_k (x^\prime_k)\right)^{1/p},\end{align*}
which is the characterization of $p$-dominated operators by means of a Pietsch-type domination that goes back to \cite{pietsch93} (see also \cite[Theorem 3.2(E)]{bprMona}). This proves the first inclusion. The remaining inclusions can be found in  \cite[Theorem 3]{daniel/erhan}.
\end{proof}

A result analogous to the item (a) above was presented in \cite[Remark 3.3]{xmtDom}  relating $\tau(p)$-summing operators and   absolutely  $(p;q_1,\ldots,$ $q_n)$-summing operators, and a result analogous to part of item (b) above was presented in \cite[Remark 5.3]{xmtDom}  relating  $\tau(p)$-summing operators and $p$-semi-integral operators.


    Now we proceed to show that bounded linear functionals on the space of $\sigma(p)$-nuclear multilinear operators are represented by quasi-$\tau(p)$-summing multilinear operators via the Borel transform.

Let us justify an equality we shall use soon: iterating the well known equality
$$ \sup_{ x^{\prime\prime} \in B_{E^{\prime\prime}}} \Bigl( \sum_{j=1}^m |x^{\prime\prime}(x^\prime_j)| \Bigr)^{1/p} =  \sup_{ x \in B_{E}}  \Bigl( \sum_{j=1}^m |x^\prime_j(x)| \Bigr)^{1/p},  $$
for $x_1', \ldots, x_m' \in E'$,
it follows that
$$\sup_{{x^{\prime\prime}_i \in B_{E^{\prime\prime}_i}}\atop {y^\prime \in B_{  F^\prime }}}  \Bigl(  \sum_{j=1}^m   \bigl|     x^{\prime\prime}_1(x^\prime_{1j}) \cdots
   x^{\prime\prime}_n(x^\prime_{nj})   y^\prime ( y_{j})   \bigr|^p\Big)^{1/p} =  \sup_{{x_i \in B_{E_i}}\atop {y^\prime \in B_{  F^\prime }}}  \Bigl( \sum_{j=1}^m
|x^\prime_{1j}(x_1)  \cdots   x^\prime_{nj}(x_n) y^\prime ( y_{j}) |^p\Bigr)^{1/p},$$
for $x_{ij}' \in E_i'$ and $y_j \in F$.

\begin{teo}\label{mth} If $E_1^\prime,\ldots,E_n^\prime $ have the bounded approximation property,  then the spaces $[{\cal L}_{\sigma(p)} (E_1, \ldots,  E_n;F)]^\prime$  and $ {\cal L}_{q\tau(p)}(E_1^\prime,   \ldots,  E_n^\prime ;F^\prime )   $
 are  isometrically isomorphic via the Borel transform, regardless of the Banach space $F$. In particular, if $E'$ has the bounded approximation property, then $[{\cal L}_{\sigma(p)}(E;F)]' \stackrel{1}{=} {\cal L}_{q\tau(p)}(E';F')$.
 \end{teo}

\begin{proof}
Given $ \varphi \in [  {\cal L}_{\sigma(p)} (E_1, \ldots,  E_n;F)]^\prime $, let us denote ${\cal B}(\varphi)$ by $ S_\varphi $.

In order to prove that  $ S_\varphi  \in    {{ {\cal L}_{q\tau}}}  _{(p)}(E_1^\prime,   \ldots,  E_n^\prime ;F^\prime )$, let $m \in  {\mathbb N}$,  $ x^\prime_{i1},
\ldots,   x^\prime_{im} \in E_i^\prime $,  $y_1, \ldots, y_m \in F$,    $ \ i=1,\ldots,n$, be given. By duality $(\ell_p^m)' = \ell_{p'}^m$ and the Hahn-Banach Theorem, there are scalars $\varepsilon_1, \ldots, \varepsilon_m$ such that $\| (\varepsilon_j)_{j=1}^m \|_{p^\prime } = 1$ and
$$\Bigl( \sum_{j=1}^m   \bigl| \varphi(x^\prime_{1j}  \otimes   \cdots \otimes   x^\prime_{nj} \otimes y_{j})  \bigr|^p\Bigr)^{1/p} = \Bigl| \sum_{j=1}^m  \varepsilon_j  \varphi(x^\prime_{1j}  \otimes   \cdots
  \otimes     x^\prime_{nj} \otimes y_{j})  \Bigr|. $$
So,
\begin{align*} \Bigl(  \sum_{j=1}^m  \bigl|  S_\varphi (x^\prime_{1j}, \ldots,     x^\prime_{nj}) (y_j)  \bigr|
^p\Bigr)&^{1/p}      =      \Bigl( \sum_{j=1}^m   \bigl| \varphi(x^\prime_{1j}  \otimes   \cdots \otimes   x^\prime_{nj} \otimes y_{j})  \bigr|^p\Bigr)^{1/p}\\
 &
 =   \Bigl| \sum_{j=1}^m  \varepsilon_j  \varphi(x^\prime_{1j}  \otimes   \cdots
  \otimes     x^\prime_{nj} \otimes y_{j})  \Bigr| \\
& =   \Bigl|  \varphi \Bigl(  \sum_{j=1}^m  \varepsilon_j   x^\prime_{1j}  \otimes   \cdots
  \otimes     x^\prime_{nj} \otimes  y_{j}  \Bigr)  \Bigr| \\
& \leq  \| \varphi\| \cdot  \Bigl\|
 \sum_{j=1}^m   \varepsilon_j   x^\prime_{1j}  \otimes   \cdots
  \otimes     x^\prime_{nj} \otimes  y_{j}   \Bigr\|_{{\sigma(p)}}\\
& \leq   \|  \varphi\| \cdot \|(\varepsilon_j)_{j=1}^m \|_{p^\prime }\cdot
  \sup_{{x_i \in B_{E_i}}\atop {y^\prime \in B_{  F^\prime }}}  \Bigl( \sum_{j=1}^m
|x^\prime_{1j}(x_1)  \cdots   x^\prime_{nj}(x_n) y^\prime ( y_{j}) |^p\Bigr)^{1/p}\\
 & =   \| \varphi\| \cdot \sup_{{x^{\prime\prime}_i \in B_{E^{\prime\prime}_i}}\atop {y^\prime \in B_{  F^\prime }}}  \Bigl(  \sum_{j=1}^m   \bigl|     x^{\prime\prime}_1(x^\prime_{1j}) \cdots
   x^{\prime\prime}_n(x^\prime_{nj})   y^\prime ( y_{j})   \bigr|^p\Big)^{1/p},
 \end{align*}
 proving that $ S_\varphi $  is  quasi-$\tau(p)$-summing  and
$\| S_\varphi \|_{q\tau(p)} \leq \| \varphi\|$.

 Conversely, given $S \in {\cal L}_{q\tau(p)}(E^\prime_1, \ldots, E^\prime_n;F^\prime)$,   define
  $$ T_S\colon   E^\prime_1 \times \cdots \times E^\prime_n \times F     \longrightarrow   {\mathbb K}
, \ T_S( x^\prime_1 , \ldots , x^\prime_n , y ) :=  S ( x^\prime_1 , \ldots , x^\prime_n ) (y ) .$$
 It is plain that $T_S$ is $(n+1)$-linear, so, having in mind that $E^\prime_1 \otimes \cdots \otimes E^\prime_n   \otimes  F = {\cal L}_f ( E_1, \ldots , E_n ; F  )$,  by the universal property of the tensor product there exists a linear operator
${\cal T_S}\colon   {\cal L}_f ( E_1, \ldots , E_n ; F  )   \longrightarrow  {\mathbb K} $ such that
$${\cal T_S}( x^\prime_1\otimes \cdots \otimes x^\prime_n \otimes y ) =   T_S( x^\prime_1 , \ldots , x^\prime_n , y ) =  S ( x^\prime_1 , \ldots , x^\prime_n ) (y ) ,$$
for all $x_1' \in E_1', \ldots, x_n' \in E_n', y \in F$.
Now we shall prove  that $ {\cal T_S} $ is continuous with respect to the norm $\|\cdot\|_{\sigma(p)}$. Given $\varepsilon > 0$ and $A \in {\cal L}_f ( E_1, \ldots , E_n ; F  )$, by definition of the norm $\|\cdot\|_{\sigma(p)f}$ we can choose a representation $A =  \sum\limits_{j=1}^m  \lambda_j   x^\prime_{1j}  \otimes   \cdots \otimes   x^\prime_{nj}  \otimes   y_j $ such that
$$ \| \bigl( \lambda_j \bigr)_{j=1}^m \|_{p^\prime} \cdot  \sup_{{x_i \in B_{E_i}}\atop {y^\prime \in B_{  F^\prime }}}    \Bigl( \sum_{j=1}^m
|x^\prime_{1j}(x_1)\cdots. x^\prime_{nj}(x_n)  y^\prime (y_j)|^p\Bigr)^{1/p}  \leq (1 + \varepsilon) \| A \|_{\sigma(p)f}.
$$
Therefore,
\begin{eqnarray*}
  | {\cal T_S} (A) |
   & = &  \Bigl|   {\cal T_S}    \Bigl(  \sum\limits_{j=1}^m \lambda_j x^\prime_{1j}\otimes\cdots\otimes x^\prime_{nj}\otimes y_j     \Bigr)   \Bigr|  =  \Bigl|    \sum\limits_{j=1}^m \lambda_j   S    \bigl( x^\prime_{1j} , \ldots , x^\prime_{nj}  \bigr)  ( y_j  )  \Bigr|  \\
    & \leq & \|(\lambda_j)_{j=1}^m\|_{p^\prime } \cdot  \Bigl( \sum_{j=1}^m   |  S(x^\prime_{1j},  \ldots,     x^\prime_{nj})(y_j)  |^p\Bigr)^{1/p}\\
          & \leq & \|(\lambda_j)_{j=1}^m\|_{p^\prime } \cdot \|S\|_{q\tau(p)} \cdot \sup_{{x^{\prime\prime}_i \in B_{E^{\prime\prime}_i}}\atop {y^\prime \in B_{  F^\prime }}}
\Bigl(   \sum\limits_{j=1}^m  | x^{\prime\prime}_1(x^\prime_{1j})  \cdots .x^{\prime\prime}_n(x^\prime_{nj})  y^\prime (y_j) |^p \Bigr)^{1/p}\\
  & = &   \|S\|_{q\tau(p)}\cdot \|(\lambda_j)_{j=1}^m\|_{p^\prime } \cdot \sup_{{x_i \in B_{E_i}}\atop {y^\prime \in B_{  F^\prime }}}
\Bigl(  \sum\limits_{j=1}^m   |x^\prime_{1j}(x_1)\cdots x^\prime_{nj}(x_n)  y^\prime (y_j) |^p \Bigr)^{1/p}\\
& \leq &    \|S\|_{q\tau(p)}  (1 + \varepsilon) \| A \|_{\sigma(p)f}.
 \end{eqnarray*}
As this holds for arbitrary $\varepsilon > 0$ and the spaces $E_1', \ldots, E_n'$ have the bounded approximation property, by Proposition \ref{6.4.3} we conclude that
$$ | {\cal T_S}  ( A) |   \leq \|S\|_{q\tau(p)}
\cdot \|A\|_{{\sigma(p)f}} = \|S\|_{q\tau(p)}
\cdot \|A\|_{{\sigma(p)}}.$$
So, ${\cal T_S} \in [{\cal L}_f ( E_1, \ldots , E_n ; F  ), \|\cdot\|_{\sigma(p)}]'$ and
$\| {\cal T_S} \|  \leq \|S\|_{q\tau(p)}$.
As ${\cal L}_f (E_1, \ldots,  E_n;F)$ is $\|\cdot\|_{\sigma(p)}$-dense in ${\cal L}_{\sigma(p)} (E_1, \ldots,  E_n;F)$ , there is a unique norm-preserving continuous linear extension $\varphi_S$ of $\cal T_S$ to the whole of ${\cal L}_{\sigma(p)} (E_1, \ldots,  E_n;F)$.  In particular, $\| \varphi_S \|  \leq \|S\|_{q\tau(p)}$ and for
  $A =  \sum\limits_{j=1}^\infty  \lambda_j   x^\prime_{1j} \otimes \cdots \otimes   x^\prime_{nj}  \otimes   y_j   \in   {\cal L}_{\sigma(p)} (E_1,\ldots, E_n; F)$,
\begin{align*}
 \varphi_S (A)
 & =    \varphi_S \Bigl( \sum_{j=1}^\infty  \lambda_j   x^\prime_{1j}  \otimes   \cdots \otimes   x^\prime_{nj}  \otimes   y_j  \Bigr)    =   \sum_{j=1}^\infty  \lambda_j   \varphi_S (  x^\prime_{1j}  \otimes   \cdots \otimes   x^\prime_{nj}  \otimes   y_j)   \\
       & =    \sum_{j=1}^\infty  \lambda_j   {\cal T_S} (  x^\prime_{1j}  \otimes   \cdots \otimes   x^\prime_{nj}  \otimes   y_j)
 =    \sum_{j=1}^\infty  \lambda_j  S(  x^\prime_{1j},  \ldots , x^\prime_{nj} )(  y_j)   .
 \end{align*}
From the expression above it follows easily that the correspondences $ \varphi \mapsto  S_\varphi $  and  $S \mapsto  \varphi_S $ are each other's inverse in the sense that $\varphi_{S_\varphi} = \varphi$ and $S_{\varphi_S} = S$ for $ \varphi\in [  {\cal L}_{\sigma(p)} (E_1, $ $ \ldots, E_n;F)]^\prime $ and $S \in
{{ {\cal L}_{q\tau(p)}}}(E_1^\prime, \ldots, $ $E_n^\prime;F^\prime )$.
The equality  $\| S_\varphi \|_{q\tau(p)} = \| \varphi\|$ completes the proof.
\end{proof}

For reflexive target spaces we have:

\begin{cor}\label{finalcor} If  $E_1^\prime,\ldots,E_n^\prime $  have the
bounded approximation property and  $F$ is reflexive, then the spaces
$ [  {\cal L}_{\sigma(p)} (E_1, \ldots,  E_n;F)]^\prime$ and  ${\cal L}_{\tau(p)}(E_1^\prime,   \ldots,  E_n^\prime ;F^\prime ) $
are isometrically isomorphic via the Borel transform. In particular, if $E'$ has the bounded approximation property and $F$ is reflexive, then $[{\cal L}_{\sigma(p)}(E;F)]' \stackrel{1}{=} {\cal L}_{\tau(p)}(E';F')$.\end{cor}

\noindent{\bf Open problem.} Is it true that ${\cal L}_{\tau(p)}(E;F') = {\cal L}_{q\tau(p)}(E;F')$ for all Banach spaces $E$ and $F$?

\bigskip

\noindent
\begin{tabular}{l}
Geraldo Botelho\\
Faculdade de Matem\'atica\\
Universidade Federal de Uberl\^andia\\
38.400-902 -- Uberl\^andia -- Brazil\\
e-mail: botelho@ufu.br\end{tabular}
\hspace{20mm}
\begin{tabular}{l}
Ximena Mujica\\
Departamento de Matem\'atica\\
Universidade Federal do Paran\'a\\
81.531-980 -- Curitiba -- Brazil\\
e-mail: xmujica@ufpr.br
\end{tabular}

\end{document}